\documentclass{amsart} 
\usepackage{amssymb, amsmath}
\usepackage[dvips]{graphicx}
\usepackage{amsfonts}
\usepackage{latexsym}
\usepackage{color}
\newtheorem{theorem}{Theorem}	
\newtheorem{lemma}{Lemma}[section]		
\newtheorem{corollary}{Corollary}		
\newtheorem{proposition}{Proposition}		
			
\newtheorem{definition}{Definition}

\title[Maximum and minimum of support functions]
{
Maximum and minimum of support functions 
}
\author{
Huhe Han}
\address{College of Science, Northwest Agriculture and Forestry University, China}
\email{han-huhe@nwafu.edu.cn}

\begin{document}
\begin{abstract}
For given continuous functions $\gamma_{{}_{i}}: S^{n}\to \mathbb{R}_{+}$  (where $i=1, 2$),
the functions $\gamma_{{}_{max}}$ and $\gamma_{{}_{min}}$ 
can be defined as natural way. 
In this paper, we show that the Wulff shape associated to  
$\gamma_{{}_{max}}$ is the convex hull of the union of Wulff shapes associated to $\gamma_{{}_1}$ and $\gamma_{{}_2}$ , if $\gamma_{{}_1}$ and $\gamma_{{}_2}$ are convex integrands. 
And, the Wulff shape associated to  
$\gamma_{{}_{min}}$ is the intersection of Wulff shapes associated to $\gamma_{{}_1}$ and $\gamma_{{}_2}$. 
Moreover, relationships between their dual Wulff shapes are given.
\end{abstract}
\maketitle  
\section{Introduction}\label{section 1}
\par
Let $n$ be a positive integer. 
Given a continuous function $\gamma: S^n\to \mathbb{R}_+$,  
where $S^n$ is the unit sphere in $\mathbb{R}^{n+1}$ and 
$\mathbb{R}_+$ is the set consisting of positive real numbers,    
the {\it Wulff shape} associated associated with the support function $\gamma$, denoted by 
$\mathcal{W}_\gamma$, is the following intersection 
(see Figure \ref{wulffshape}), 
\[
\mathcal{W}_\gamma=\bigcap_{\theta\in S^n}\Gamma_{\gamma, \theta}.  
\]
Here, $\Gamma_{\gamma, \theta}$ is the following half-space:   
\[
\Gamma_{\gamma, \theta}=\{x\in \mathbb{R}^{n+1}\; |\; x\cdot \theta\le \gamma(\theta)\}, 
\]    
where the dot in the center stands for the dot product of two vectors 
$x, \theta\in \mathbb{R}^{n+1}$.     
\begin{figure}[htbp]
  \begin{center}
    \includegraphics[clip,width=4.0cm]{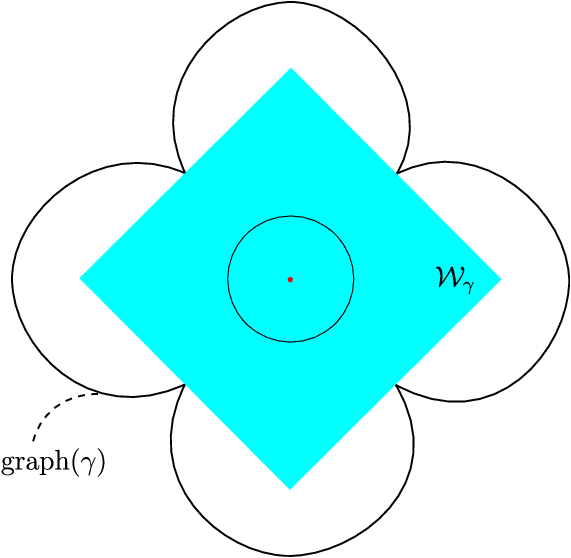}
    \caption{A Wulff shape $\mathcal{W}_{\gamma}$.}
    \label{wulffshape}
  \end{center}
\end{figure}
This construction is well-known as Wulff's construction of an equilibrium  crystal  introduced 
by G.~Wulff in \cite{wulff}      
(for details on Wulff shapes, see for instance \cite{giga, morgan, crystalbook, taylor, taylor2}). 
By definition, a Wulff shape is convex, compact and it contains the origin of $\mathbb{R}^{n+1}$ as an interior point. 
Conversely, it has been known that any convex body $W$ in $\mathbb{R}^{n+1}$ containing the origin as an interior point is a Wulff shape associated with an appropriate support function, namely, 
there exists a continuous function $\gamma : S^n \to \mathbb{R}_+$ such that $W = \mathcal{W}_{\gamma}$ 
(\cite{taylor}).
\par 
For a continuous function 
$\gamma: S^n\to \mathbb{R}_+$,    
set 
\[
\mbox{graph}(\gamma)=\{(\theta, \gamma(\theta))\in \mathbb{R}^{n+1}-\{0\}\; |\; \theta\in S^n\}, 
\]
where $(\theta, \gamma(\theta))$ is the polar plot expression for a point of $\mathbb{R}^{n+1}-\{0\}$.   
The mapping 
$\mbox{inv} : \mathbb{R}^{n+1}-\{0\}\to \mathbb{R}^{n+1}-\{0\}$, defined as follows, is called the {\it inversion} with respect to 
the origin of $\mathbb{R}^{n+1}$.     
\[
\mbox{inv}(\theta, r)=\left(-\theta, \frac{1}{r}\right). 
\]
Let $\Gamma_\gamma$ be the boundary of the convex hull of 
$\mbox{inv}(\mbox{graph}(\gamma))$.  
If the equality 
$\Gamma_\gamma=\mbox{inv}(\mbox{graph}(\gamma))$ is satisfied, then 
$\gamma$ is called a \textit{convex integrand} (see Figure \ref{convexintegrand}). 
\begin{figure}[htbp]
  \begin{center}
    \includegraphics[clip,width=5.0cm]{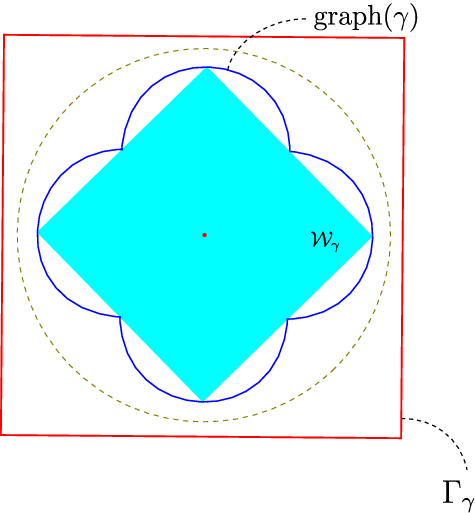}
    \caption{A convex integrand and its inversion.}
    \label{convexintegrand}
  \end{center}
\end{figure}
The notion of convex integrand was first introduced by J.~Taylor in \cite{taylor} 
and  it plays a key role for studying Wulff shapes (for details on convex integrands, 
see for instance \cite{ hannishimura2, morgan}).   
For given support functions $\gamma_{{}_{1}}, \gamma_{{}_{2}}$,  
define $\gamma_{{}_{max}}$ and $\gamma_{{}_{min}}$ as follows.
\[
\gamma_{{}_{max}}: S^{n}\to \mathbb{R}_{+}, \gamma_{{}_{max}}(\theta)={\rm max}\{\gamma_{{}_{1}}(\theta), \gamma_{{}_{2}}(\theta)\}.
\]
\[
\gamma_{{}_{min}}: S^{n}\to \mathbb{R}_{+}, \gamma_{{}_{min}}(\theta)={\rm min}\{\gamma_{{}_{1}}(\theta), \gamma_{{}_{2}}(\theta)\}.
\]
Then the question naturally arises 
``What are the relationships between 
$\mathcal{W}_{\gamma_{{}_{1}}}, \mathcal{W}_{\gamma_{{}_{2}}}$ and $\mathcal{W}_{\gamma_{{}_{max}}}$ (or $\mathcal{W}_{\gamma_{{}_{min}}}$)? How they are related ?''. 
The main results of this paper are as follows (see Figure \ref{figure 3}).
\begin{theorem}\label{maximumoperator}
Let $\gamma_{{}_{1}}$ and $\gamma_{{}_{2}}$ be convex integrands. 
Then the following holds:
\[
\mathcal{W}_{\gamma_{{}_{max}}}={\rm convex\ hull\ of\ the}\ (\mathcal{W}_{\gamma_{{}_{1}}}\cup \mathcal{W}_{\gamma_{{}_{2}}}).
\]
\end{theorem} 
\begin{theorem}\label{minimumoperator}
Let $\gamma_{{}_{1}}$ and $\gamma_{{}_{2}}$ be support(continuuos) functions. 
Then the following holds:
\[
\mathcal{W}_{\gamma_{{}_{min}}}= \mathcal{W}_{\gamma_{{}_{1}}}\cap \mathcal{W}_{\gamma_{{}_{2}}}.
\]
\end{theorem} 

\begin{figure}[htbp]
  \begin{center}
  \begin{minipage}{.45\linewidth} 
	\begin{center}
\includegraphics[width=0.8 \linewidth]{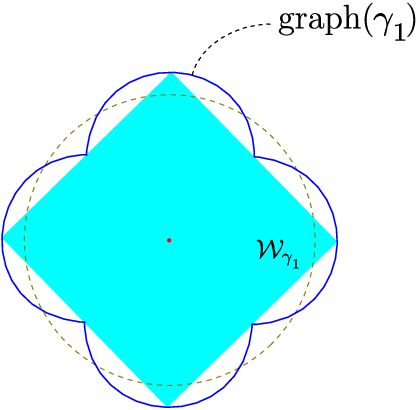} 
	\end{center}
  \end{minipage}
 \hspace{2.5pc} 
  \begin{minipage}{.45\linewidth} 
  	\begin{center}
\includegraphics[width=0.8 \linewidth]{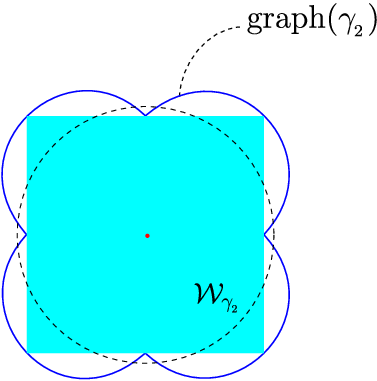} 
	\end{center}
  \end{minipage}
  \end{center}
\vspace{10mm}
  \begin{center}
  \begin{minipage}{.45\linewidth} 
  \begin{center}
  \includegraphics[width=0.9 \linewidth]{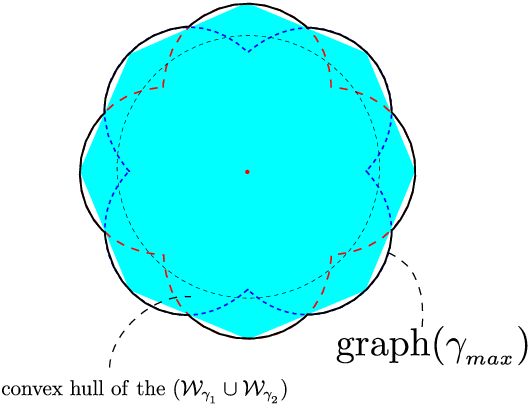} 
	\end{center}
  \end{minipage}
  \hspace{2.0pc} 
  \begin{minipage}{.45\linewidth} 
  \begin{center}
 \includegraphics[width=0.9 \linewidth]{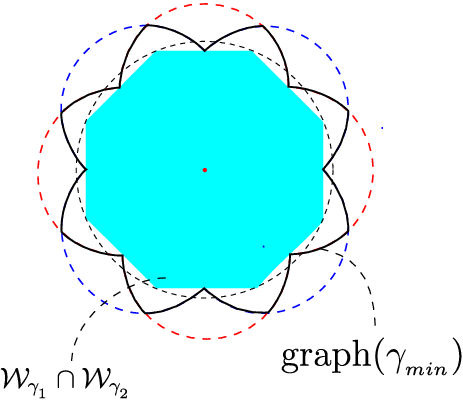} 
	\end{center}
  \end{minipage}
  \end{center}
  \caption{An illustration of Theorem \ref{maximumoperator},  \ref{minimumoperator}.    
Left top : a Wulff shape $\mathcal{W}_{\gamma_{{}_{1}}}$, right top : a Wulff shape $\mathcal{W}_{\gamma_{{}_{2}}}$, left bottom :  the Wulff shape $\mathcal{W}_{\gamma_{{}_{max}}}$, right bottom : the Wulff shape $\mathcal{W}_{\gamma_{{}_{min}}}$.}
  \label{figure 3}
  \end{figure}
\par 
\medskip 
This paper is organized as follows. In section 2, the preliminaries are given, 
and the proof of Theorem \ref{maximumoperator}, Theorem \ref{minimumoperator} and related topics are given
in section 3, 4 and section 5 respectively.


\section{Preliminaries}\label{section 2}
\subsection{Spherical convex body}
For any point $\widetilde{P} \in S^{n+1}$, let $H(\widetilde{P})$
be the hemisphere centered at $\widetilde{P}$,
\[
H(\widetilde{P})=\{\widetilde{Q}\in S^{n+1}\mid \widetilde{P}\cdot \widetilde{Q}\geq 0\},
\]
where the dot in the center stands for the scalar product of two vectors 
$\widetilde{P}, \widetilde{Q}\in \mathbb{R}^{n+2}$.     
For any {\color{black}non-empty} subset $\widetilde{W}\subset S^{n+1}$, the {\it spherical polar set of $\widetilde{W}$}, denoted by 
$\widetilde{W}^\circ$, is defined as follows: 
\[
\widetilde{W}^\circ = \bigcap_{\widetilde{P}\in \widetilde{W}}H(\widetilde{P}).
\]   
\begin{definition}[\cite{nishimurasakemi2}]\label{definition 1}
\rm{Let $\widetilde{W}$ be a subset of $S^{n+1}$.     
Suppose that there exists a point $\widetilde{P}\in S^{n+1}$ such that 
$\widetilde{W}\cap H(\widetilde{P})=\emptyset$. 
Then, $\widetilde{W}$ is said to be 
{\it hemispherical}.      
}
\end{definition}
{\color{black}
Let $\widetilde{P}, \widetilde{Q}$ be two points of $S^{n+1}$ such that $(1-t)\widetilde{P}+t\widetilde{Q}$ is not the zero vector 
for any $t\in [0,1]$. 
Then, 
the following arc is denoted by $\widetilde{P}\widetilde{Q}$:  
\[
\widetilde{P}\widetilde{Q}=\left\{\left.\frac{(1-t)\widetilde{P}+t\widetilde{Q}}{||(1-t)\widetilde{P}+t\widetilde{Q}||}\in S^{n+1}\; \right|\; 0\le t\le 1\right\}.  
\]
}
\begin{definition}[\cite{nishimurasakemi2}] 
\label{definition 2.2}
{\rm 
Let $\widetilde{W}\subset S^{n+1}$ be a hemispherical subset.   
\begin{enumerate}
\item Suppose that 
$\widetilde{P}\widetilde{Q}\subset \widetilde{W}$ for any 
$\widetilde{P}, \widetilde{Q}\in \widetilde{W}$. 
Then,  
$\widetilde{W}$ is said to be {\it spherical convex}.   
{\color{black}
\item Suppose that $\widetilde{W}$ is closed, spherical convex and has an interior point.   
Then,  
$\widetilde{W}$ is said to be a {\it spherical convex body}.   
}
\end{enumerate}
}
\end{definition}  
\begin{definition}[\cite{nishimurasakemi2}]\label{definition 2.2}
{\rm Let $\widetilde{W}$ be a hemispherical subset of $S^{n+1}$.     
Then, the following set, denoted by 
$\mbox{\rm s-conv}({\color{black}\widetilde{W}})$, is called the {\it spherical convex hull of} 
${\color{black}\widetilde{W}}$.   
\[
\mbox{\rm s-conv}(\widetilde{W})= 
\left\{\left.
\frac{\sum_{i=1}^k t_i\widetilde{P}_i}{||\sum_{i=1}^kt_i\widetilde{P}_i||}\;\right|\; 
\widetilde{P}_i\in {\color{black}\widetilde{W}},\; \sum_{i=1}^kt_i=1,\; t_i\ge 0, k\in \mathbb{N}
\right\}.
\] 
}
\end{definition}
\begin{lemma}[\cite{nishimurasakemi2}]\label{smallestlem}
For any hemispherical subset $\widetilde{W}$ of $S^{n+1}$, the spherical convex hull of $\widetilde{W}$ is the smallest spherical convex set containing $\widetilde{W}$. 
\end{lemma}
\begin{lemma}[\cite{nishimurasakemi2}]\label{inclusionlem}
Let $\widetilde{W}_{1}, \widetilde{W}_{2} \subset S^{n+1}$. 
If $\widetilde{W}_{1}$ is a subset of 
$\widetilde{W}_{2}$,  
then $\widetilde{W}_{2}^{\circ}$ is a subset of $\widetilde{W}_{1}^{\circ}$.
\end{lemma}
\begin{lemma}[\cite{nishimurasakemi2}]\label{inclusionitself}
For any subset $\widetilde{W}$ of $S^{n+1}$, the inclusion $\widetilde{W}\subset \widetilde{W}^{\circ \circ}$ holds.
\end{lemma}

The following proposition has been known.
\begin{proposition}[\cite{nishimurasakemi2}]\label{proposition 2.1}
For any non-empty closed hemispherical subset $\widetilde{W}\subset S^{n+1}$,    
the equality $\mbox{\rm s-conv}(\widetilde{W})=(\mbox{\rm s-conv}(\widetilde{W}))^{\circ\circ}$ holds.

\end{proposition}
\begin{lemma}[Maehara's lemma (\cite{maehara, nishimurasakemi2})]\label{maeharalemma}
For any hemispherical finite subset $\widetilde{W}=\{\widetilde{P}_1, \ldots, \widetilde{P}_k\}\subset S^{n+1}$, the following holds:
$$
\left\{\left.
\frac{\sum_{i=1}^k t_i\widetilde{P}_i}{||\sum_{i=1}^kt_i\widetilde{P}_i||}\;\right|\; 
\widetilde{P}_i\in X,\; \sum_{i=1}^kt_i=1,\; t_i\ge 0
\right\}^\circ 
= 
H(\widetilde{P}_1)\cap \cdots \cap H(\widetilde{P}_k).
$$
\label{lemma 2.5}
\end{lemma}
Maeara's lemma was first given in \cite{maehara}, which is a useful tool to study spherical convex bodies. For details, 
see for instance \cite{nishimurasakemi2}.
%
\subsection{An equivalent definition}
In \cite{nishimurasakemi2}, an equivalent definition of Wulff shape
has been given, which is defined as the composition of the following mappings.
\par
\vspace{1mm}
{\bf 1. Mapping $Id: \mathbb{R}^{n+1}\to \mathbb{R}^{n+1}\times \{1\}$} 
\par
The mapping $Id$ defined by $Id(x)=(x, 1)$.
\par
\vspace{1mm}
{\bf 2. Central projection $\alpha_{{}_{N}}: S_{N,+}^{n+1}\to \mathbb{R}^{n+1}\times \{1\}$}  

\vspace{1mm}
Denote the point $(0, \ldots, 0, 1)\in \mathbb{R}^{n+2}$ by $N$.  
The set $S^{n+1}-H(-N)$ is denoted by $S_{N,+}^{n+1}$.    
Let $\alpha_{{}_{N}}: S_{N,+}^{n+1}\to \mathbb{R}^{n+1}\times \{1\}$ be the central projection relative to 
$N$, namely, $\alpha_N$ is defined as follows for any 
$(P_1, \ldots, P_{n+1}, P_{n+2})\in S_{N, +}^{n+1}$:   
\[
\alpha_{{}_{N}}\left(P_1, \ldots, P_{n+1}, P_{n+2}\right)
=
\left(\frac{P_1}{P_{n+2}}, \ldots, \frac{P_{n+1}}{P_{n+2}}, 1\right).    
\]
Through out remainder of this paper,
let $\widetilde{X}=\alpha_{{}_{N}}^{-1}\circ Id(X)$ for any non-empty subset $X$ of $\mathbb{R}^{n+1}$. For any Wulfff shape $\mathcal{W}\subset \mathbb{R}^{n+1}$, the spherical convex body $\widetilde{\mathcal{W}}=\alpha_{{}_{N}}^{-1}\circ Id(\mathcal{W})$ 
is called {\it spherical Wulff shape of $\mathcal{W}$}.
Then the following two are equivalent (see \cite{nishimurasakemi2} for details ).   
\begin{enumerate}
\item $\widetilde{\mathcal{W}}$ is a spherical Wulff shape.
\item $\widetilde{\mathcal{W}}$ is a spherical convex body such that 
$\widetilde{\mathcal{W}}\cap H(-N)=\emptyset$ and $N$ is an interior point of $\widetilde{\mathcal{W}}$.   
\end{enumerate}

\vspace{1mm}
{\bf 3. Spherical blow-up $\Psi_N:S^{n+1}-\{\pm N\}\to {\color{black}S_{N, +}^{n+1}}$}

\vspace{1mm}
\par 
Next, we consider the mapping $\Psi_N:S^{n+1}-\{\pm N\}\to {\color{black}S_{N, +}^{n+1}}$ 
defined by 
\[
\Psi_N(\widetilde{P})=\frac{1}{\sqrt{1-(N\cdot P)^2}}(N-(N\cdot \widetilde{P})\widetilde{P}).    
\]
The mapping $\Psi_N$, which was first introduced in \cite{nishimura}, 
has the following intriguing properties:   
\begin{enumerate}
\item For any $\widetilde{P}\in S^{n+1}-\{\pm N\}$, the equality $\widetilde{P}\cdot \Psi_N(\widetilde{P})=0$ holds,    
\item for any $\widetilde{P}\in S^{n+1}-\{\pm N\}$, the property $\Psi_N(\widetilde{P})\in \mathbb{R}N+\mathbb{R}\widetilde{P}$ holds,     
\item for any $\widetilde{P}\in S^{n+1}-\{\pm N\}$, the property $N\cdot \Psi_N(\widetilde{P})>0$ holds,      
\item the restriction $\Psi_N|_{S^{n+1}_{N,+}-\{N\}}: S^{n+1}_{N,+}-\{N\}\to S^{n+1}_{N,+}-\{N\}$ is a $C^\infty$ diffeomorphism.   
\end{enumerate}

\vspace{1mm}
{\bf 4. Spherical polar transform $\bigcirc: \mathcal{H}^{\circ}(S^{n+1})\to \mathcal{H}^{\circ}(S^{n+1})$}

\par
Let $\mathcal{H}(S^{n+1})$ be the set consisting of non-empty compact set of $S^{n+1}$. It is clear that the spherical polar set of $S^{n+1}$ is the empty set. 
Let $\mathcal{H}^{\circ}(S^{n+1})$ be the subspace of $\mathcal{H}(S^{n+1})$ defined as follows.
\[
\mathcal{H}^{\circ}(S^{n+1})=\{\widetilde{W}\in \mathcal{H}(S^{n+1})\mid \widetilde{W}^{\circ}\neq \emptyset \}.
\]
The {\it spherical polar transform $\bigcirc: \mathcal{H}^{\circ}(S^{n+1})\to \mathcal{H}^{\circ}(S^{n+1})$} is defined by 
$\bigcirc(\widetilde{W})=\widetilde{W}^{\circ}.$ 
Since $\widetilde{W}\subset \widetilde{W}^{\circ \circ}$ for
any $\widetilde{W} \in \mathcal{H}^{\circ}(S^{n+1})$ 
by Lemma \ref{inclusionitself}, it follows that 
$\widetilde{W}^{\circ}\in \mathcal{H}^{\circ}(S^{n+1})$ for any 
$\widetilde{W}\in \mathcal{H}^{\circ}(S^{n+1})$. 
Thus, the spherical polar transform 
 is well-defined.
It is known that the spherical polar transform is Lipchitz with 
respect to the
Pompeiu-Hausdorff distance(\cite{hannishimura}). Moreover, 
the restriction of the spherical polar transform to the set consisting of 
spherical Wulff shape relative to $\widetilde{P}$ (see Definition \ref{sphericalwulffshape}) is an isometry 
with 
respect to the Pompeiu-Hausdorff distance (\cite{hannishimura}). 
\begin{proposition}[\cite{nishimurasakemi2}]\label{proposition 6}
Let $\gamma: S^{n}\to \mathbb{R}_+$ be a continuous function.    
Then, $\mathcal{W}_\gamma$ {\color{black}is} characterized as follows:   
\[
\mathcal{W}_\gamma = 
Id^{-1}\circ \alpha_{{}_{N}}\left(\left(\Psi_N\circ \alpha_{{}_{N}}^{-1}\circ 
Id\left(\mbox{\rm graph}(\gamma)\right)\right)^\circ\right).   
\]
\end{proposition}
Proposition \ref{proposition 6} implies Wulff shapes $\mathcal{W}_{\gamma_{{}_{1}}}$ 
and $\mathcal{W}_{\gamma_{{}_{2}}}$ are same convex bodies if and only if 
$\left(\Psi_N\circ \alpha_{{}_{N}}^{-1}\circ 
Id\left(\mbox{\rm graph}(\gamma_{{}_{1}})\right)\right)^\circ$ 
and 
$\left(\Psi_N\circ \alpha_{{}_{N}}^{-1}\circ 
Id\left(\mbox{\rm graph}(\gamma_{{}_{2}})\right)\right)^\circ$ 
are same spherical convex bodies. 
By Maehara's lemma (Lemma \ref{maeharalemma}), we know that 
\[
\left(\Psi_N\circ \alpha_{{}_{N}}^{-1}\circ 
Id\left(\mbox{\rm graph}(\gamma_{{}_{i}})\right)\right)^\circ= \left(\mbox{s-conv}
\bigl(
\Psi_N\circ \alpha_{{}_{N}}^{-1}\circ 
Id\left(\mbox{\rm graph}(\gamma_{{}_{i}})\right)
\bigr)
\right)^\circ, \]
where $i=1,2.$ 
Therefore, the following holds: 
\begin{proposition}[\cite{nishimurasakemi2}]\label{proposition 789}
Let $\gamma_{{}_{i}}: S^n \to \mathbb{R}_+$ be continuous function, where $i=1,2$.  
Then the following two statements are equivalent:
\begin{enumerate}
\item $\mathcal{W}_{\gamma_{{}_{1}}}=\mathcal{W}_{\gamma_{{}_{2}}}$.
\item 
$\mbox{\rm s-conv}
\bigl(
\Psi_N\circ \alpha_{{}_{N}}^{-1}\circ 
Id\left(\mbox{\rm graph}(\gamma_{{}_{1}})\right)
\bigr)
=
\mbox{\rm s-conv}
\bigl(
\Psi_N\circ \alpha_{{}_{N}}^{-1}\circ 
Id\left(\mbox{\rm graph}(\gamma_{{}_{2}})\right)
\bigr)$.
\end{enumerate}
\end{proposition}
\begin{proposition}[\cite{nishimurasakemi2}]\label{proposition 7}
For any Wulff shape $\mathcal{W}_\gamma$, 
the following set, {\color{black}too,} is a Wulff shape: 
\[
Id^{-1}\circ \alpha_{{}_{N}}\left(\left(\alpha_{{}_{N}}^{-1}\circ 
Id\left(\mathcal{W}_\gamma\right)\right)^\circ\right).  
\] 
\end{proposition}

\begin{definition}[\cite{nishimurasakemi2}]
{\rm 
For any Wulff shape $\mathcal{W}_\gamma$, the Wulff shape 
{\color{black}given} in Proposition {\color{black}\ref{proposition 7}} is called the {\it dual Wulff shape} of 
$\mathcal{W}_\gamma$ and is denoted by $\mathcal{D}\mathcal{W}_\gamma$.  
} 
\end{definition}
{\color{black}
\noindent 
By Proposition \ref{proposition 7}, the following definition is reasonable.   
\begin{definition}[\cite{nishimurasakemi2}]\label{sphericalwulffshape}
{\rm 
Let $\widetilde{P}$ be a point of $S^{n+1}$.   
\begin{enumerate}
\item {\color{black}A} spherical convex body 
$\widetilde{W}$ such that $\widetilde{W}\cap H(-\widetilde{P})=\emptyset$ 
and $\widetilde{P}\in \mbox{\rm int}(\widetilde{W})$ are satisfied is called 
a {\it spherical Wulff shape} relative to $\widetilde{P}$.     
\item Let $\widetilde{\mathcal{W}}$ be a spherical Wulff shape relative to $\widetilde{P}$.    
Then, the set $\widetilde{\mathcal{W}}^\circ$ 
is called the {\it spherical dual Wulff shape} of the spherical Wulff shape 
$\widetilde{\mathcal{W}}$ relative to $\widetilde{P}$ and is denoted by $\mathcal{D}\widetilde{\mathcal{W}}$.   
\end{enumerate}
}
\end{definition}
}
\begin{proposition}[\cite{nishimurasakemi2}]\label{dualandboundary}
Let $\bar{\gamma}:S^{n}\to \mathbb{R}_{+}$ be a convex integrand. 
Let $\mathcal{D}\mathcal{W}_{\gamma}$ be the dual Wulff shape of $\mathcal{W}_{\gamma}.$ 
Then the boundary of the $\mathcal{D}\mathcal{W}_{\gamma}$ is exactly $\Gamma_\gamma=\mbox{inv}(\mbox{graph}(\gamma)).$
\end{proposition}
Proposition \ref{proposition 6} gives a new powerful spherical method to study Wulff shapes. For example, the self-dual Wulff shape $\mathcal{W}_{\gamma}=\mathcal{DW}_{\gamma}$ 
can be characterized by the induced spherical convex body of constant width ${\pi}/{2}$ (\cite{hannishimura3}). More details on width of spherical convex bodies and their duals, see for instance 
\cite{hanwu}--\cite{michal}.
For related topics on the spherical method see for instance \cite{hannishimura4}. 
\section{Proof of Theorem \ref{maximumoperator}}\label{section 3}
By Proposition \ref{proposition 6}, $\mathcal{W}_{\gamma_{{}_{max}}}$ can be rewritten as
\begin{alignat*}{3}
\mathcal{W}_{\gamma_{{}_{max}}}&= Id^{-1}\circ \alpha_{{}_{N}}\left( \left(\Psi_{N}\circ \alpha_{{}_{N}}^{-1}\circ Id({\rm graph}(\gamma_{{}_{max}}))\right)^{\circ}\right)\notag \\
&=  Id^{-1}\circ \alpha_{{}_{N}} \left((\partial (\mathcal{D}\widetilde{\mathcal{W}}_{\gamma_{{}_{1}}}\cap \mathcal{D}\widetilde{\mathcal{W}}_{\gamma_{{}_{2}}}))^{\circ}\right)\\
&= Id^{-1}\circ \alpha_{{}_{N}}\left( ( \mathcal{D}\widetilde{\mathcal{W}}_{\gamma_{{}_{1}}}\cap \mathcal{D}\widetilde{\mathcal{W}}_{\gamma_{{}_{2}}})^{\circ}\right) 
\end{alignat*}
where $\partial  (\mathcal{D}\widetilde{\mathcal{W}}_{\gamma_{{}_{1}}}\cap \mathcal{D}\widetilde{\mathcal{W}}_{\gamma_{{}_{2}}})$ 
is the boundary of $ (\mathcal{D}\widetilde{\mathcal{W}}_{\gamma_{{}_{1}}}\cap \mathcal{D}\widetilde{\mathcal{W}}_{\gamma_{{}_{2}}})$, and 
\[
\widetilde{\mathcal{W}}_{\gamma_{{}_{i}}}=\alpha_{{}_{N}}^{-1}\circ Id(\mathcal{W}_{\gamma_{{}_{i}}})
\]
 is the spherical Wulff shape of 
$\mathcal{W}_{\gamma_{{}_{i}}}, i= 1, 2$. 
Here, the second equality follows from Proposition \ref{dualandboundary} and 
the third equality follows from Maehara's lemma (Lemma \ref{maeharalemma}).
Thus it is sufficient to prove the following:
\[
(\mathcal{D}\widetilde{\mathcal{W}}_{\gamma_{{}_{1}}}\cap \mathcal{D}\widetilde{\mathcal{W}}_{\gamma_{{}_{2}}})^{\circ} 
= 
\mbox{s-conv}
(\widetilde{\mathcal{W}}_{\gamma_{{}_{1}}}\cup \widetilde{\mathcal{W}}_{\gamma_{{}_{2}}}). 
\]
\indent
First, we show that 
$( \mathcal{D}\widetilde{\mathcal{W}}_{\gamma_{{}_{1}}}\cap \mathcal{D}\widetilde{\mathcal{W}}_{\gamma_{{}_{2}}})^{\circ}\subset 
\mbox{s-conv}(\widetilde{\mathcal{W}}_{\gamma_{{}_{1}}}\cup \widetilde{\mathcal{W}}_{\gamma_{{}_{2}}})$.
Let $\widetilde{P}$ be a point of 
$(\widetilde{\mathcal{W}}_{\gamma_{{}_{1}}}\cup \widetilde{\mathcal{W}}_{\gamma_{{}_{2}}})^{\circ}$. 
Then it follows that
\[
\widetilde{\mathcal{W}}_{\gamma_{{}_{1}}}\subset H(\widetilde{P})
\ {\rm and}\ 
\widetilde{\mathcal{W}}_{\gamma_{{}_{2}}}\subset H(\widetilde{P}).
\]
Thus $\widetilde{P}$ is a point of 
$\widetilde{\mathcal{W}}_{\gamma_{{}_{1}}}^{\circ}\cap 
\widetilde{\mathcal{W}}_{\gamma_{{}_{2}}}^{\circ}=\mathcal{D}\widetilde{\mathcal{W}}_{\gamma_{{}_{1}}}\cap \mathcal{D}\widetilde{\mathcal{W}}_{\gamma_{{}_{2}}}$. 
This implies that
\[
(\widetilde{\mathcal{W}}_{\gamma_{{}_{1}}}\cup \widetilde{\mathcal{W}}_{\gamma_{{}_{2}}})^{\circ}
\subset
\widetilde{\mathcal{W}}_{\gamma_{{}_{1}}}^{\circ}\cap 
\widetilde{\mathcal{W}}_{\gamma_{{}_{2}}}^{\circ}
=\mathcal{D}\widetilde{\mathcal{W}}_{\gamma_{{}_{1}}}\cap \mathcal{D}\widetilde{\mathcal{W}}_{\gamma_{{}_{2}}}.
\]
Then by Maehara's lemma, we have
\[ 
\bigl(\mbox{s-conv}(\widetilde{\mathcal{W}}_{\gamma_{{}_{1}}}\cup \widetilde{\mathcal{W}}_{\gamma_{{}_{2}}})\bigr)^{\circ}
=
(\widetilde{\mathcal{W}}_{\gamma_{{}_{1}}}\cup \widetilde{\mathcal{W}}_{\gamma_{{}_{2}}})^{\circ}
\subset
\mathcal{D}\widetilde{\mathcal{W}}_{\gamma_{{}_{1}}}\cap \mathcal{D}\widetilde{\mathcal{W}}_{\gamma_{{}_{2}}}.
\]
Since $\bigl(\mbox{s-conv}(\widetilde{\mathcal{W}}_{\gamma_{{}_{1}}}\cup \widetilde{\mathcal{W}}_{\gamma_{{}_{2}}})\bigr)^{\circ\circ}
=
\mbox{s-conv}(\widetilde{\mathcal{W}}_{\gamma_{{}_{1}}}\cup \widetilde{\mathcal{W}}_{\gamma_{{}_{2}}})$ 
(Proposition \ref{proposition 2.1}), by Lemma \ref{inclusionlem}, it follows that
\[
\bigl(\mathcal{D}\widetilde{\mathcal{W}}_{\gamma_{{}_{1}}}\cap \mathcal{D}\widetilde{\mathcal{W}}_{\gamma_{{}_{2}}}\bigr)^\circ
\subset 
\bigl(\mbox{s-conv}(\widetilde{\mathcal{W}}_{\gamma_{{}_{1}}}\cup \widetilde{\mathcal{W}}_{\gamma_{{}_{2}}})\bigr)^{\circ\circ}
=
\mbox{s-conv}(\widetilde{\mathcal{W}}_{\gamma_{{}_{1}}}\cup \widetilde{\mathcal{W}}_{\gamma_{{}_{2}}}).
\] 
\par
\medskip
\indent
Next, we show that  
$\mbox{s-conv}(\widetilde{\mathcal{W}}_{\gamma_{{}_{1}}}\cup \widetilde{\mathcal{W}}_{\gamma_{{}_{2}}})\subset 
( \mathcal{D}\widetilde{\mathcal{W}}_{\gamma_{{}_{1}}}\cap \mathcal{D}\widetilde{\mathcal{W}}_{\gamma_{{}_{2}}})^{\circ}$. 
Since $\mathcal{D}\widetilde{\mathcal{W}}_{\gamma_{{}_{1}}}\cap \mathcal{D}\widetilde{\mathcal{W}}_{\gamma_{{}_{2}}}$ is a subset of 
$\mathcal{D}\widetilde{\mathcal{W}}_{\gamma_{{}_{1}}}$, 
by Lemma \ref{inclusionlem}, we have  
\[
(\mathcal{D}\widetilde{\mathcal{W}}_{\gamma_{{}_{1}}}\cap \mathcal{D}\widetilde{\mathcal{W}}_{\gamma_{{}_{2}}})\subset \widetilde{W}_{\gamma_{{}_{1}}}^{\circ}. 
\]
Then by Proposition \ref{proposition 2.1}, it follows that
\[
\widetilde{W}_{\gamma_{{}_{1}}}=\widetilde{W}_{\gamma_{{}_{1}}}^{\circ\circ}
\subset (\mathcal{D}\widetilde{\mathcal{W}}_{\gamma_{{}_{1}}}\cap \mathcal{D}\widetilde{\mathcal{W}}_{\gamma_{{}_{2}}})^{\circ}.
\]
In the same way, the following inclusion is holds:
\[
\widetilde{W}_{\gamma_{{}_{2}}}
\subset (\mathcal{D}\widetilde{\mathcal{W}}_{\gamma_{{}_{1}}}\cap \mathcal{D}\widetilde{\mathcal{W}}_{\gamma_{{}_{2}}})^{\circ}. 
\]
Since $\mbox{s-conv}(\widetilde{\mathcal{W}}_{\gamma_{{}_{1}}}\cup \widetilde{\mathcal{W}}_{\gamma_{{}_{2}}})$ 
is the smallest convex body containing 
$\widetilde{\mathcal{W}}_{\gamma_{{}_{1}}}\cup \widetilde{\mathcal{W}}_{\gamma_{{}_{2}}}$ (Lemma \ref{smallestlem}) and
$(\mathcal{D}\widetilde{\mathcal{W}}_{\gamma_{{}_{1}}}\cap \mathcal{D}\widetilde{\mathcal{W}}_{\gamma_{{}_{2}}})^{\circ}$ 
is a convex body, it follows that 
\[
\mbox{s-conv}
(\widetilde{\mathcal{W}}_{\gamma_{{}_{1}}}\cup \widetilde{\mathcal{W}}_{\gamma_{{}_{2}}})\subset 
(\mathcal{D}\widetilde{\mathcal{W}}_{\gamma_{{}_{1}}}\cap \mathcal{D}\widetilde{\mathcal{W}}_{\gamma_{{}_{2}}})^{\circ}. 
\] 
\indent
This proves the Theorem. \hfill{$\square$}
\medskip
\\
\indent
\par
Remark that the condition $\lq\lq$convex integrand" of Theorem \ref{maximumoperator} is necessary. 
Theorem \ref{maximumoperator} does not hold in general (see Figure \ref{counter1} and \ref{counter2}).
\begin{figure}[htbp]  \begin{center}
    \includegraphics[clip,width=5.0cm]{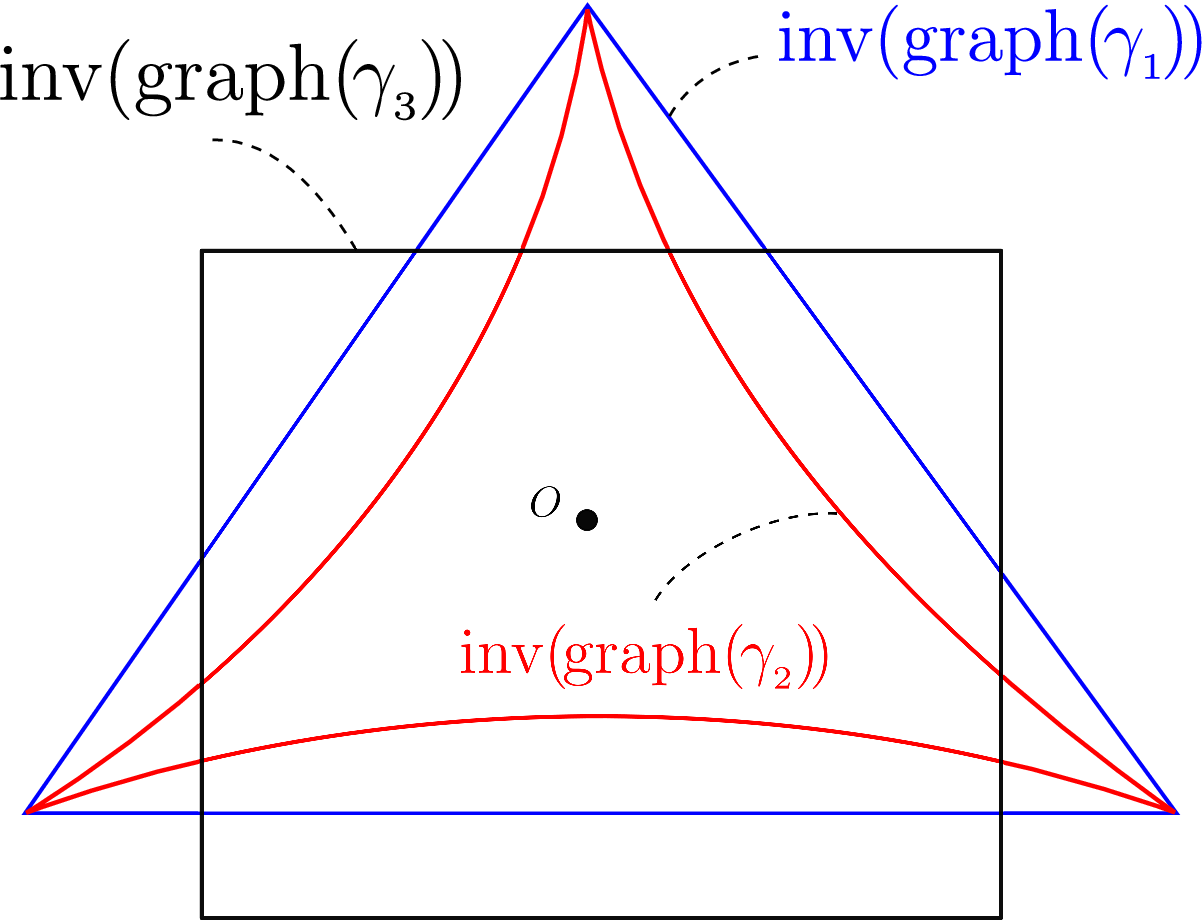}
    \caption{Boundary of the Wulff shape $\mathcal{DW}_{\gamma_{{}_{i}}}$, where $i=1,2,3$ and $\mathcal{DW}_{\gamma_{{}_{1}}}= \mathcal{DW}_{\gamma_{{}_{2}}}$.}
    \label{counter1}
 \end{center}
\end{figure}

\begin{figure}[htbp]
  \begin{center}
    \includegraphics[clip,width=13cm]{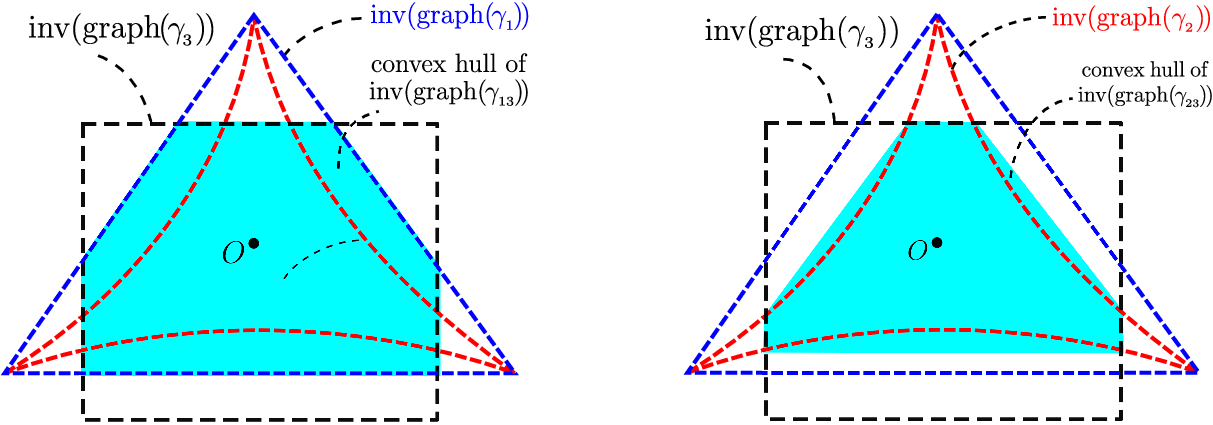}
    \caption{$\mathcal{DW}_{\gamma_{{}_{13}}}\not= \mathcal{DW}_{\gamma_{{}_{23}}}$, where
    $\gamma_{{}_{ij}}$ is defined by the maximum of the support functions $\gamma_{{}_{i}}, \gamma_{{}_{j}}$.}
\label{counter2}
  \end{center}
\end{figure}
\par
\indent
Usually, for given convex integrands 
$\gamma_{{}_{1}}$ and $\gamma_{{}_{2}}$, the inversion of 
$\mbox{graph}(\gamma_{{}_{min}})$ with respect to the origin of $\mathbb{R}^{n+1}$ does not necessarily equal to $\Gamma_{\gamma_{{}_{min}}}$. 
Thus the function 
$\gamma_{{}_{min}}$ is not convex integrand in general. 
On the other hand, the inversion of $\mbox{graph}(\gamma_{{}_{max}})$ with respect to the origin of $\mathbb{R}^{n+1}$ is exactly $\Gamma_{\gamma_{{}_{max}}}$, which implies that $\gamma_{{}_{max}}$ is a convex integrand. Thus we can generalize Theorem \ref{maximumoperator} as follows.
\begin{corollary}
Let $\gamma_{{}_{1}}, \dots , \gamma_{{}_{m}}$ be convex integrands. 
Let $\widetilde{\gamma}_{{}_{max}}: S^{n}\to \mathbb{R}_{+}$ be the function  defined as $\widetilde{\gamma}_{{}_{max}}(\theta)={\rm max}\{\gamma_{{}_1}(\theta), \dots , \gamma_{{}_m}(\theta)\}.$ Then the Wulff shape of $\widetilde{\gamma}_{{}_{max}}$ is the convex hull of $\mathcal{W}_{\gamma_{{}_{1}}}\cup \dots \cup \mathcal{W}_{\gamma_{{}_{m}}}$.
\end{corollary}
\section{Proof of Theorem \ref{minimumoperator}}\label{section 4}
We first prove that the equality
$
\mathcal{W}_{\gamma_{{}_{min}}}= \mathcal{W}_{\gamma_{{}_{1}}}\cap \mathcal{W}_{\gamma_{{}_{2}}}
$
holds for any convex integrands $\gamma_{{}_1}, \gamma_{{}_2}$. 
\begin{lemma}\label{minconvexintegrand}
Let $\bar{\gamma}_{{}_1}, \bar{\gamma}_{{}_2}$ be convex integrands such that 
$\mathcal{W}_{\gamma_{{}_{1}}}=\mathcal{W}_{\bar{\gamma}_{{}_{1}}}, \mathcal{W}_{\gamma_{{}_{2}}}=\mathcal{W}_{\bar{\gamma}_{{}_{2}}}$. 
Set $\bar{\gamma}_{{}_{min}}: S^n\to \mathbb{R}_+$ is the continuous function defined by 
$\bar{\gamma}_{{}_{min}}(\theta)= {\rm min}\{\bar{\gamma}_{{}_{1}}(\theta), \bar{\gamma}_{{}_{2}}(\theta)\}.$
Then $\mathcal{W}_{\bar{\gamma}_{{}_{min}}}= \mathcal{W}_{\bar{\gamma}_{{}_{1}}}\cap \mathcal{W}_{\bar{\gamma}_{{}_{2}}}$.
\end{lemma}
\underline{Proof of Lemma \ref{minconvexintegrand}.} 
By Proposition \ref{proposition 6}, $\mathcal{W}_{\bar{\gamma}_{{}_{min}}}$ can be rewritten as
\begin{alignat*}{3} 
\mathcal{W}_{\bar{\gamma}_{{}_{min}}}&= Id^{-1}\circ \alpha_{N}\left( (\Psi_{N}\circ \alpha_{{}_{N}}^{-1}\circ Id({\rm graph}(\bar{\gamma}_{{}_{min}})))^{\circ}\right)\notag \\
&=  Id^{-1}\circ \alpha_{{}_{N}}\left( (\partial (\mathcal{D}\widetilde{\mathcal{W}}_{\bar{\gamma}_{{}_{1}}}\cup \mathcal{D}\widetilde{\mathcal{W}}_{\bar{\gamma}_{{}_{2}}}))^{\circ}\right)\\
&= Id^{-1}\circ \alpha_{{}_{N}}\left( ( \mathcal{D}\widetilde{\mathcal{W}}_{\bar{\gamma}_{{}_{1}}}\cup \mathcal{D}\widetilde{\mathcal{W}}_{\bar{\gamma}_{{}_{2}}})^{\circ}\right).
\end{alignat*}
Here, the second equality follows from Proposition \ref{dualandboundary} and 
the thrid equality follows from Maehara's lemma.
Thus it is sufficient to prove the following:
\[
        (\mathcal{D}\widetilde{\mathcal{W}}_{\bar{\gamma}_{{}_{1}}}\cup \mathcal{D}\widetilde{\mathcal{W}}_{\bar{\gamma}_{{}_{2}}})^{\circ} 
= 
        \widetilde{\mathcal{W}}_{\bar{\gamma}_{{}_{1}}}\cap \widetilde{\mathcal{W}}_{\bar{\gamma}_{{}_{2}}}. 
\]
\indent
First, we show that 
\[
(\mathcal{D}\widetilde{\mathcal{W}}_{\bar{\gamma}_{{}_{1}}}\cup \mathcal{D}\widetilde{\mathcal{W}}_{\bar{\gamma}_{{}_{2}}})^{\circ}
                \subset  
\bigl(
\widetilde{\mathcal{W}}_{\bar{\gamma}_{{}_{1}}}\cap \widetilde{\mathcal{W}}_{\bar{\gamma}_{{}_{2}}}
\bigr).
\]
Let $\widetilde{P}\in (\mathcal{D}\widetilde{\mathcal{W}}_{\bar{\gamma}_{{}_{1}}}\cup \mathcal{D}\widetilde{\mathcal{W}}_{\bar{\gamma}_{{}_{2}}})^{\circ}$. 
Then it follows that $\mathcal{D}\widetilde{\mathcal{W}}_{\bar{\gamma}_{{}_{1}}}\subset  H(\widetilde{P})$. 
By Proposition \ref{proposition 789}, we have
\[
\widetilde{P}\in (\mathcal{D}\widetilde{\mathcal{W}}_{\bar{\gamma}_{{}_{1}}})^{\circ}=\widetilde{\mathcal{W}}_{\bar{\gamma}_{{}_{1}}}^{\circ\circ}=\widetilde{\mathcal{W}}_{\bar{\gamma}_{{}_{1}}}.
\]
\indent
In the same way, we have that
$\widetilde{P}\in\widetilde{\mathcal{W}}_{\bar{\gamma}_{{}_{2}}}$. 
Thus $\widetilde{P}\in \widetilde{\mathcal{W}}_{\bar{\gamma}_{{}_{1}}}\cap \widetilde{\mathcal{W}}_{\bar{\gamma}_{{}_{2}}}$. \\
\par
\medskip
Next, we show that 
\[
\bigl(
\widetilde{\mathcal{W}}_{\bar{\gamma}_{{}_{1}}}\cap \widetilde{\mathcal{W}}_{\bar{\gamma}_{{}_{2}}}
\bigr)
\subset 
(\mathcal{D}\widetilde{\mathcal{W}}_{\bar{\gamma}_{{}_{1}}}\cup \mathcal{D}\widetilde{\mathcal{W}}_{\bar{\gamma}_{{}_{2}}})^{\circ}.
\]
 Let $\widetilde{P}$ be a point of  $\widetilde{\mathcal{W}}_{\bar{\gamma}_{{}_{1}}}\cap \widetilde{\mathcal{W}}_{\bar{\gamma}_{{}_{2}}}$. 
Since $\widetilde{\mathcal{W}}_{\bar{\gamma}_{{}_{1}}}$ and 
$\widetilde{\mathcal{W}}_{\bar{\gamma}_{{}_{2}}}$ 
are spherical convex bodies, by Proposition \ref{proposition 2.1}, it follows that 
\[
\widetilde{P}\in \widetilde{\mathcal{W}}_{\bar{\gamma}_{{}_{i}}}
=
\bigl(
\mathcal{D}\widetilde{\mathcal{W}}_{\bar{\gamma}_{{}_{i}}}
\bigr)^\circ
=\bigcap_{\widetilde{Q}\in \mathcal{D}\widetilde{\mathcal{W}}_{\bar{\gamma}_{{}_{i}}}}H(\widetilde{Q}), 
\]
where $i=1, 2$.
This implies
\[
 \mathcal{D}\widetilde{\mathcal{W}}_{\bar{\gamma}_{{}_{1}}}\subset H(\widetilde{P})\ {\rm and}\ \mathcal{D}\widetilde{\mathcal{W}}_{\bar{\gamma}_{{}_{2}}}\subset H(\widetilde{P}).
\]
Then it follows that
$\mathcal{D}\widetilde{\mathcal{W}}_{\bar{\gamma}_{{}_{1}}}\cup \mathcal{D}\widetilde{\mathcal{W}}_{\bar{\gamma}_{{}_{2}}}$
is a subset of $H(\widetilde{P})$.
By the definition of spherical polar set, 
$\widetilde{P}$ is a point of 
$(\mathcal{D}\widetilde{\mathcal{W}}_{\bar{\gamma}_{{}_{1}}}\cup \mathcal{D}\widetilde{\mathcal{W}}_{\bar{\gamma}_{{}_{2}}})^{\circ}$. 
Therefore, 
$
\widetilde{\mathcal{W}}_{\bar{\gamma}_{{}_{1}}}\cap \widetilde{\mathcal{W}}_{\bar{\gamma}_{{}_{2}}}
$ 
is a subset of
$(\mathcal{D}\widetilde{\mathcal{W}}_{\bar{\gamma}_{{}_{1}}}\cup \mathcal{D}\widetilde{\mathcal{W}}_{\bar{\gamma}_{{}_{2}}})^{\circ}$. 
\hfill{$\square$} \\
\par
\indent
The next lemma is useful in the coming proof.
\begin{lemma} \label{lemma42}
The following inclusion is holds.
\[
\bigl(
\mbox{\rm s-conv}\left(\Psi_{N}\circ \alpha_{{}_{N}}^{-1}\circ Id({\rm graph}(\gamma_{{}_{min}}))\right)
\bigr)^\circ
\subset
\bigl(
\mbox{\rm s-conv} \left( \Psi_{N}\circ \alpha_{{}_{N}}^{-1}\circ Id({\rm graph}(\bar{\gamma}_{{}_{min}}))\right)
\bigr)^\circ.
\]
\end{lemma}
\noindent
\underline{Proof of Lemma \ref{lemma42}.}
Since $\gamma_{{}_{min}}(\theta)\leq \gamma_{{}_{1}}(\theta)$, for any $\theta\in S^n$, it follows that 
\[
\mid\mid (\theta, \gamma_{{}_{min}}(\theta)) \mid\mid\leq \mid\mid (\theta, \gamma_{{}_{1}}(\theta)) \mid\mid.
\] 
Set 
\[
\widetilde{R}_1=\alpha_{{}_{N}}^{-1}\circ Id((\theta, \gamma_{{}_{1}}(\theta))\ 
\mbox{and}\ 
\widetilde{R}_{min}=\alpha_{{}_{N}}^{-1}\circ Id((\theta, \gamma_{{}_{min}}(\theta)).
\]
By the properties $(1)-(3)$ of spherical blow-up $\Psi_N$, we have
\begin{alignat*}{3}
\mid\mid N \bigl(\Psi_{N}(\widetilde{R}_1)\bigr)\mid\mid
&= \frac{\pi}{2}-\mid\mid N \widetilde{R}_1\mid\mid\notag \\
&\leq  \frac{\pi}{2}-\mid\mid N \widetilde{R}_{min}\mid\mid\notag \\
&=\mid\mid N \bigl(\Psi_{N}( \widetilde{R}_{min})\bigr)\mid\mid,  \notag
\end{alignat*}
for any $\theta\in S^n$ (see Figure \ref{sphericalblowup}). 
\begin{figure}[htbp]\label{sphericalblowup}
  \begin{center}
    \includegraphics[clip,width=7.0cm]{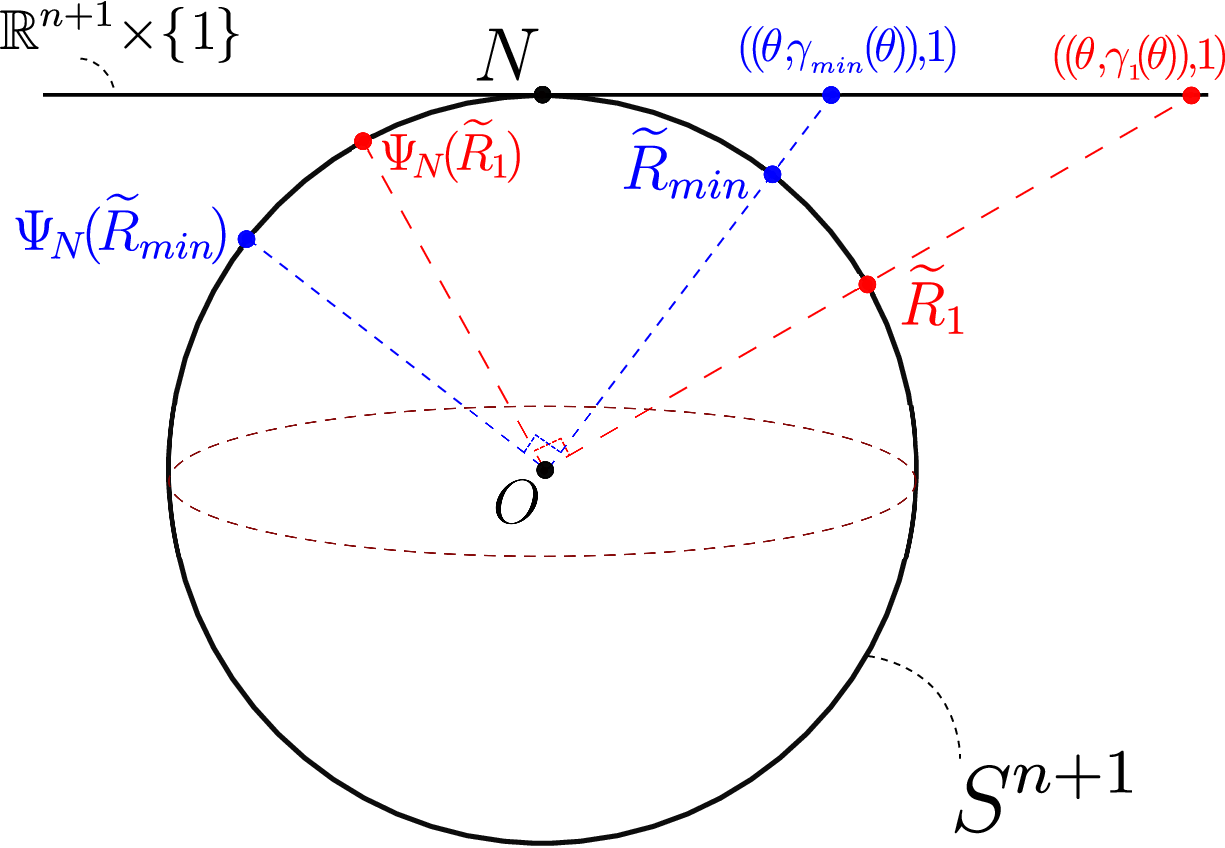}
    \caption{
    $\mid\mid N \bigl(\Psi_{N}( \widetilde{R})\bigr)\mid\mid
    \leq 
    \mid\mid N \bigl(\Psi_{N}(\widetilde{R})\bigr)\mid\mid$
    .}
    \label{convexintegrand}
  \end{center}
\end{figure}
Thus it follows that 
\[
\tag{$1$} 
\mbox{s-conv}\left(\Psi_{N}\circ \alpha_{{}_{N}}^{-1}\circ Id({\rm graph}(\gamma_{{}_{1}}))\right)
\subset 
\mbox{s-conv}\left(\Psi_{N}\circ \alpha_{{}_{N}}^{-1}\circ Id({\rm graph}(\gamma_{{}_{min}}))\right).
\]
Since $\bar{\gamma}_{{}_{1}}$ is the convex integrand of  $\mathcal{W}_{\gamma_{{}_{1}}}$, 
by Proposition \ref{proposition 789}, we have 
\[\tag{$2$} 
\mbox{s-conv}\left(\Psi_{N}\circ \alpha_{{}_{N}}^{-1}\circ Id({\rm graph}(\bar{\gamma}_{{}_{1}}))\right)
=
\mbox{s-conv}\left(\Psi_{N}\circ \alpha_{{}_{N}}^{-1}\circ Id({\rm graph}(\gamma_{{}_{1}}))\right).
\]
Putting ($1$) together with ($2$), we conclude that 
\[\tag{$3$} 
\mbox{s-conv}\left(\Psi_{N}\circ \alpha_{{}_{N}}^{-1}\circ Id({\rm graph}(\bar{\gamma}_{{}_{1}}))\right)
\subset 
\mbox{s-conv}\left(\Psi_{N}\circ \alpha_{{}_{N}}^{-1}\circ Id({\rm graph}(\gamma_{{}_{min}}))\right).
\]
In the same way, it follows that
\[\tag{$4$} 
\mbox{s-conv}\left(\Psi_{N}\circ \alpha_{{}_{N}}^{-1}\circ Id({\rm graph}(\bar{\gamma}_{{}_{2}}))\right)
\subset 
\mbox{s-conv}\left(\Psi_{N}\circ \alpha_{{}_{N}}^{-1}\circ Id({\rm graph}(\gamma_{{}_{min}}))\right).
\]
Because $\Psi_{N}\circ \alpha_{{}_{N}}^{-1}\circ Id({\rm graph}(\bar{\gamma}_{{}_{i}}))$ 
is a subset of 
$\mbox{s-conv}\left(\Psi_{N}\circ \alpha_{{}_{N}}^{-1}\circ Id({\rm graph}(\bar{\gamma}_{{}_{i}}))\right)$, 
$i= 1, 2,$ it follows that
\begin{alignat*}{3}
&\mbox{s-conv} \left( \Psi_{N}\circ \alpha_{{}_{N}}^{-1}\circ Id({\rm graph}(\bar{\gamma}_{{}_{min}}))\right)\\
&= \mbox{s-conv}\Bigl(
\cup_{i=1}^2 \bigl(\Psi_{N}\circ \alpha_{{}_{N}}^{-1}\circ Id({\rm graph}(\bar{\gamma}_{{}_{i}})\bigr)\Bigr)\notag \\
&\subset  \mbox{s-conv}
\Bigl(
\cup_{i=1}^2
\mbox{s-conv}\bigl(\Psi_{N}\circ \alpha_{{}_{N}}^{-1}\circ Id({\rm graph}(\bar{\gamma}_{{}_{i}}))\bigr) \Bigr)\\
&\subset \mbox{s-conv}\left(\Psi_{N}\circ \alpha_{{}_{N}}^{-1}\circ Id({\rm graph}(\gamma_{{}_{min}}))\right)\\
\end{alignat*}
Here, the equality follows from definition of $\bar{\gamma}_{{}_{min}}$, 
the last inclusion follows from ($3$), ($4$) and
\[
\mbox{s-conv}
\Bigl(
\cup_{i=1}^2
\mbox{s-conv}\bigl(\Psi_{N}\circ \alpha_{{}_{N}}^{-1}\circ Id({\rm graph}(\gamma_{{}_{i}}))\bigr) \Bigr)
\]
is the smallest spherical convex set containing 
\[
\cup_{i=1}^2
\mbox{s-conv}\bigl(\Psi_{N}\circ \alpha_{{}_{N}}^{-1}\circ Id({\rm graph}(\gamma_{{}_{i}}))\bigr).
\]
Then by Lemma \ref{inclusionlem}, we conclude that
\[
\bigl(
\mbox{s-conv}\left(\Psi_{N}\circ \alpha_{{}_{N}}^{-1}\circ Id({\rm graph}(\gamma_{{}_{min}}))\right)
\bigr)^\circ
\subset
\bigl(
\mbox{s-conv} \left( \Psi_{N}\circ \alpha_{{}_{N}}^{-1}\circ Id({\rm graph}(\bar{\gamma}_{{}_{min}}))\right)
\bigr)^\circ.
\]
\hfill{$\square$} 
\par
\indent
We are now in the position to show the equality
$
\mathcal{W}_{\gamma_{{}_{min}}}= \bigl(
\mathcal{W}_{\gamma_{{}_{1}}}\cap \mathcal{W}_{\gamma_{{}_{2}}}
\bigr)$
holds for any support functions $\gamma_{{}_1}, \gamma_{{}_2}$. 
\bigskip
\par
\indent
By Lemma \ref{minconvexintegrand}, it is sufficient to prove that 
\[
\mathcal{W}_{\bar{\gamma}_{{}_{min}}}=\mathcal{W}_{\gamma_{{}_{min}}}.
\]
Since $\bar{\gamma}_{{}_1}(\theta) \leq \gamma_{{}_1}(\theta)$ and $\bar{\gamma}_{{}_2}(\theta) \leq \gamma_{{}_2}(\theta)$ for any $\theta\in S^n$, we have taht
\[
\bar{\gamma}_{{}_{min}}(\theta)
=
{\rm min}\{\bar{\gamma}_{1}(\theta), \bar{\gamma}_{2}(\theta)\} 
\leq 
{\rm min}\{\gamma_{1}(\theta), \gamma_{2}(\theta)\}
=
\gamma_{{}_{min}}(\theta).
\]
This implies 
\[
\mathcal{W}_{\bar{\gamma}_{{}_{min}}}\subset \mathcal{W}_{\gamma_{{}_{min}}}.
\] 
So it is sufficient to prove that 
\[
\mathcal{W}_{\gamma_{{}_{min}}} \subset \mathcal{W}_{\bar{\gamma}_{{}_{min}}}.
\]
By Proposition \ref{proposition 6}, this follows from following:
\begin{alignat*}{3}
\mathcal{W}_{\gamma_{{}_{min}}}&= Id^{-1}\circ \alpha_{N}\left( (\Psi_{N}\circ \alpha_{{}_{N}}^{-1}\circ Id({\rm graph}(\gamma_{{}_{min}})))^{\circ}\right)\notag \\
&=  
Id^{-1}\circ \alpha_{N}
\left(\bigl( \mbox{s-conv}\left(\Psi_{N}\circ \alpha_{{}_{N}}^{-1}\circ Id({\rm graph}(\gamma_{{}_{min}}))\right)
\bigr)^\circ\right)\\
&\subset 
Id^{-1}\circ \alpha_{N}\left( 
(\mbox{s-conv} \left( \Psi_{N}\circ \alpha_{{}_{N}}^{-1}\circ Id({\rm graph}(\bar{\gamma}_{{}_{min}}))\right)
\bigr)^\circ
\right)\\
&= 
Id^{-1}\circ \alpha_{N}\bigl( 
( \Psi_{N}\circ \alpha_{{}_{N}}^{-1}\circ Id({\rm graph}(\bar{\gamma}_{{}_{min}})))
)^\circ
\bigr)\\
&= 
\mathcal{W}_{\bar{\gamma}_{{}_{min}}}. 
\end{alignat*}
Here, the second and forth equalities follows from Maehara's lemma, 
the inclusion follows from Lemma \ref{lemma42}.
This proves the Theorem.
\hfill{$\square$} 
\begin{corollary}
Let $\gamma_{{}_{1}}, \dots , \gamma_{{}_{m}}$ be support functions. 
Let $\widetilde{\gamma}_{{}_{min}}: S^{n}\to \mathbb{R}_{+}$ be the function  defined as $\widetilde{\gamma}_{{}_{min}}(\theta)={\rm min}\{\gamma_{{}_1}(\theta), \dots , \gamma_{{}_m}(\theta)\}.$ Then the Wulff shape of $\widetilde{\gamma}_{{}_{min}}$ is the intersection $\mathcal{W}_{\gamma_{{}_{1}}}\cap \dots \cap \mathcal{W}_{\gamma_{{}_{m}}}$.
\end{corollary}
\section{More topics on maximum and minimum of convex integrands}
For given convex integrands (resp. continuous functions) $\gamma_{{}_{1}}$ and $\gamma_{{}_{2}}$, 
by proof of Theorem \ref{maximumoperator} (resp. Theorem \ref{minimumoperator}), 
we have the following relation between  
$\mathcal{DW}_{\gamma_{{}_{1}}}$, $\mathcal{DW}_{\gamma_{{}_{2}}}$ 
and $\mathcal{W}_{\gamma_{{}_{max}}}$
(resp. $\mathcal{DW}_{\gamma_{{}_{1}}}$, $\mathcal{DW}_{\gamma_{{}_{2}}}$ 
and $\mathcal{W}_{\gamma_{{}_{min}}}$):
\[
\mathcal{W}_{\gamma_{{}_{max}}}=\mathcal{D}(\mathcal{DW}_{\gamma_{{}_{1}}}\cap \mathcal{DW}_{\gamma_{{}_{2}}})
\]
\[ 
\left(
\mbox{resp.}\ 
\mathcal{W}_{\gamma_{{}_{min}}}=\mathcal{D}\bigl(\mbox{convex hull of }  (\mathcal{DW}_{\gamma_{{}_{1}}}\cup \mathcal{DW}_{\gamma_{{}_{2}}})\bigr)
\right). 
\]
Since $\mathcal{DW}_{\gamma_{{}_{1}}}\cap \mathcal{DW}_{\gamma_{{}_{2}}}$ is a subset of the
$\mbox{convex hull of } \mathcal{DW}_{\gamma_{{}_{1}}}\cup \mathcal{DW}_{\gamma_{{}_{2}}}$, 
by Lemma \ref{inclusionlem}, it follows that 
\[
\bigl(
\mbox{s-conv }
(\mathcal{D}\widetilde{\mathcal{W}}_{\gamma_{{}_{1}}}\cup \mathcal{D}\widetilde{\mathcal{W}}_{\gamma_{{}_{2}}})
\bigr)^{\circ} 
\subset
(\mathcal{D}\widetilde{\mathcal{W}}_{\gamma_{{}_{1}}}\cap \mathcal{D}\widetilde{\mathcal{W}}_{\gamma_{{}_{2}}})^{\circ}.
\] 
\indent
Moreover, as a corollary, we have the following:
\begin{corollary}
Let $\gamma_{{}_{1}}$ and $\gamma_{{}_{2}}$ be convex integrands. 
Suppose that $\mathcal{W}_{\gamma_{{}_{1}}}$ is the dual Wulff shape of 
$\mathcal{W}_{\gamma_{{}_{2}}}$. 
Then $\mathcal{W}_{\gamma_{{}_{max}}}$ 
is the dual Wulff shape of $\mathcal{W}_{\gamma_{{}_{min}}}$. 
\end{corollary}
\begin{proof}
By Theorem \ref{maximumoperator} and \ref{minimumoperator}, we know that
\begin{alignat*}{3}
\mathcal{W}_{\gamma_{{}_{max}}}
&= 
Id^{-1}\circ \alpha_{{}_{N}}
\bigl( \mbox{s-conv}
(\widetilde{\mathcal{W}}_{\gamma_{{}_{1}}}\cup \widetilde{\mathcal{W}}_{\gamma_{{}_{2}}})
\bigr),
\notag \\
\mathcal{W}_{\gamma_{{}_{min}}}
&=
Id^{-1}\circ \alpha_{{}_{N}}
\bigl(
 \widetilde{\mathcal{W}}_{\gamma_{{}_{1}}}\cap \widetilde{\mathcal{W}}_{\gamma_{{}_{2}}}
\bigr).
  \notag
\end{alignat*}
Since Maehara's lemma implies
 \[
 (\widetilde{\mathcal{W}}_{\gamma_{{}_{1}}}\cup \widetilde{\mathcal{W}}_{\gamma_{{}_{2}}})^{\circ}
 = 
 \bigl(
 \mbox{s-conv}(\widetilde{\mathcal{W}}_{\gamma_{{}_{1}}}\cup \widetilde{\mathcal{W}}_{\gamma_{{}_{2}}})
\bigr)^{\circ}, 
\] 
it is sufficient to prove that 
\[
(\widetilde{W}_{\gamma_{{}_{1}}}\cup \widetilde{W}_{\gamma_{{}_{2}}})^{\circ}
= \bigl(
\widetilde{W}_{\gamma_{{}_{1}}} \cap \widetilde{W}_{\gamma_{{}_{2}}}
\bigr).
 \] 
Let $\widetilde{P}$ be a point of $(\widetilde{W}_{\gamma_{{}_{1}}}\cup \widetilde{W}_{\gamma_{{}_{2}}})^{\circ}$. 
Then it follows that
\[
\bigl(
\widetilde{W}_{\gamma_{{}_{1}}}\cup \widetilde{W}_{\gamma_{{}_{2}}} 
\bigr)
\subset  H(\widetilde{P}).
\] 
This implies $\widetilde{P}$ is a point of $\widetilde{W}_{\gamma_{{}_{1}}}^{\circ}\cap \widetilde{W}_{\gamma_{{}_{2}}}^{\circ}$. 
Since $\widetilde{W}_{\gamma_{{}_{1}}}$ is the dual of $\widetilde{W}_{\gamma_{{}_{2}}}$, namely,
\[
\widetilde{W}_{\gamma_{{}_{1}}}^\circ=\widetilde{W}_{\gamma_{{}_{2}}}, \widetilde{W}_{\gamma_{{}_{1}}}=\widetilde{W}_{\gamma_{{}_{2}}}^\circ,
\]
it follows that 
\[
\widetilde{P}\in 
\bigl(
\widetilde{W}_{\gamma_{{}_{1}}}^{\circ}\cap \widetilde{W}_{\gamma_{{}_{2}}}^{\circ}
\bigr)
=
\bigl(
\widetilde{W}_{\gamma_{{}_{2}}}\cap \widetilde{W}_{\gamma_{{}_{1}}}
\bigr).
\]
Therefore, we conclude that
\[
(\widetilde{W}_{\gamma_{{}_{1}}}\cup \widetilde{W}_{\gamma_{{}_{2}}})^{\circ}
\subset
 \bigl(
\widetilde{W}_{\gamma_{{}_{1}}} \cap \widetilde{W}_{\gamma_{{}_{2}}}
\bigr).
 \] 
\bigskip
\par
On the other hand, since Wulff shapes $\mathcal{W}_{\gamma_{{}_{1}}}, \mathcal{W}_{\gamma_{{}_{2}}}$ are 
duals, we have 
\[\tag{$5$}
\bigl(
\widetilde{W}_{\gamma_{{}_{1}}}\cap \widetilde{W}_{\gamma_{{}_{2}}}
\bigr)
=\bigl(
\widetilde{W}_{\gamma_{{}_{2}}}^\circ\cap \widetilde{W}_{\gamma_{{}_{1}}}^\circ
\bigr).
\]
Let $\widetilde{P}\in \bigl(
\widetilde{W}_{\gamma_{{}_{1}}}\cap \widetilde{W}_{\gamma_{{}_{2}}}
\bigr)$.
By ($5$), it follows that 
\[
\bigl(
\widetilde{W}_{\gamma_{{}_{1}}}\cup \widetilde{W}_{\gamma_{{}_{2}}}
\bigr)
\subset
H(\widetilde{P}).
\]
Then we have 
\[
\widetilde{P}\in (\widetilde{W}_{\gamma_{{}_{1}}}\cup \widetilde{W}_{\gamma_{{}_{2}}})^{\circ}.
\] 
Therefore, it follows that
\[
\bigl(
\widetilde{W}_{\gamma_{{}_{1}}} \cap \widetilde{W}_{\gamma_{{}_{2}}}
\bigr)
\subset
(\widetilde{W}_{\gamma_{{}_{1}}}\cup \widetilde{W}_{\gamma_{{}_{2}}})^{\circ}.
\]
\end{proof}
\section*{Acknowledgements}
The author would like to express his sincere appreciation to Takashi Nishimura, for his kind advice. 
This work was partially supported by
Natural Science Basic Research Program of Shaanxi (Program No. 2020JQ-235)
and the Initial Foundation for Scientific Research
of Northwest A\&F University (Program No. 2452018018).

\medskip


\begin{thebibliography}{99}          
\bibitem{giga}Y.~Giga, 
\textit{Surface Evolution Equations}, Monographs of Mathematics, {\bf 99}, Springer, 2006.   
\bibitem{hannishimura}H.~Han and T.~Nishimura, 
\textit{The spherical dual transform is an isometry for spherical Wulff shapes}, Studia Math., {\bf 245} (2019), 201-211.  
\bibitem{hannishimura2}H.~Han and T.~Nishimura, 
\textit{Strictly convex Wulff shapes and $C^1$ convex integrands}, 
Proc. Amer. Math. Soc., {\bf 145} (2017), 3997-4008. 
\bibitem{hannishimura3}H.~Han and T.~Nishimura, 
\textit{Self dual Wulff shapes and spherical convex bodies of constant width ${\pi}/{2}$}, 
 J. Math. Soc. Japan., {\bf 69} (2017), 1475-1484.  
\bibitem{hannishimura4}H.~Han and T.~Nishimura, 
\textit{Spherical method for studying Wulff shapes and related topics}, 
Singularities in generic geometry, 1-53, Adv. Stud. Pure Math.,{\bf 78}, Math. Soc. Japan, Tokyo, 2018.

\bibitem{hanwu}H.~Han and D.~Wu, 
\emph{Constant diameter and constant width of spherical convex bodies},
Aequationes Math., to appear.

\bibitem{Lassak15}M.~Lassak,
\emph{Width of spherical convex bodies},
Aequationes Math., {\bf 89} (2015), 555--567.

\bibitem{LM18}M.~Lassak and M.~Musielak,
\emph{Spherical bodies of constant width},
Aequationes Math., {\bf 92} (2018), 627--640.

\bibitem{Lassak19}M.~Lassak,
\emph{When a spherical body of constant diameter is of constant width?},
Aequationes Math., {\bf 94} (2020), 393--400.


\bibitem{michal}M. Musielak,
\emph{Covering a reduced spherical body by a disk},	
Ukrain Math. J., to appear.

\bibitem{maehara}H.~Maehara, {\it Geometry of Circles and Spheres}, Asakura Publishing, 1998 (in Japanese). 
\bibitem{morgan}F.~Morgan, 
\textit{The cone over the Clifford torus 
in $\mathbb{R}^4$ is $\Phi$-minimizing}, {Math.~Ann.},  {\bf 289} (1991), 
341--354.
\bibitem{nishimura}T.~Nishimura, 
\textit{Normal forms for singularities of pedal curves produced by non-singular 
dual curve germs in $S^n$},  
{Geom Dedicata} {\bf 133}(2008), 59--66.     
\bibitem{nishimurasakemi2} T.\ Nishimura and Y.\ Sakemi, 
\textit{Topological aspect of Wulff shapes}, 
J. Math. Soc. Japan, {\bf 66} (2014), 89--109.
\bibitem{crystalbook}A.~Pimpinelli and J.~Villain, 
{\it Physics of Crystal Growth}, Monographs and Texts in Statistical Physics, 
Cambridge University Press, Cambridge New York, 1998. 
\bibitem{taylor}J.~E.~Taylor, \textit{Crystalline variational problems}, Bull. Amer. Math. Soc., 
{\bf 84}(1978), 568--588.     
\bibitem{taylor2}J.~E.~Taylor, J.~W.~Cahn and C.~A.~Handwerker, 
\textit{Geometric models of crystal growth}, 
Acta Metallurgica et Materialia, {\bf 40}(1992), 1443--1474.  
\bibitem{wulff}G.~Wulff, 
\textit{Zur frage der geschwindindigkeit 
des wachstrums und der aufl\"osung der krystallflachen}, 
Z. Kristallographine und Mineralogie, {\bf 34}(1901), 449--530.
\end{thebibliography}
\end{document}